\documentclass[a4paper,12pt,twoside]{amsart}
\usepackage{amssymb}

\usepackage{amsfonts,amssymb,amscd, amsmath}
\usepackage{graphicx}
 \usepackage[left=2.3cm,top=3.3cm,right=2.3cm,bottom=2.7cm]{geometry}
\usepackage{mathrsfs}

\usepackage[all,ps,cmtip]{xy}

\def\llra{\hbox to 10mm{\rightarrowfill}}

\def\lllra{\hbox to 15mm{\rightarrowfill}}

\def\PB{{\widehat B}}
\def\PK{{\widehat K}}

\def\PT{{\widehat T}}
\def\PC{{\widehat C}}

\def\phi{{\varphi}}
\def\wf{{\widetilde f}}

\def\wY{{\widetilde Y}}

\def\cI{\mathcal{I}}
\def\cD{\mathcal{D}}

\def\cF{\mathcal{F}}

\def\cO{\mathcal{O}}

\def\cP{\mathcal{P}}
\def\cH{\mathcal{H}}

\def\cS{\mathcal{S}}

\def\cM{\mathcal{M}}

\def\cV{\mathcal{V}}

\let\tilde\widetilde

\DeclareMathOperator{\Pic}{Pic}

\newtheorem{lemm}{Lemma}[section]
\newtheorem{theo}[lemm]{Theorem}
\newtheorem{coro}[lemm]{Corollary}

\newtheorem{prop}[lemm]{Proposition}

\newtheorem*{conj*}{Conjecture}

\theoremstyle{definition}

\newtheorem{rema}[lemm]{Remark}

\theoremstyle{remark}
\newtheorem*{remark*}{Remark}
\newtheorem*{note*}{Note}

\begin{document}
\title[Cohomological support loci and Pluricanonical systems]{Cohomological support loci and Pluricanonical systems on irregular varieties}
\author{Zhi Jiang}
\address{Zhi Jiang, Shanghai Center for Mathematical Sciences, China}
\email{zhijiang@fudan.edu.cn}
\date{\today}
\subjclass[2010]{14E05, 14D07, 14F10}
\keywords{cohomological support loci, pluricanonical maps, generic vanishing}
\maketitle
\begin{abstract} 
For an irregular variety $X$ of general type, we show that if a general fiber $F$ of the Albanese morphism of $X$ satisfies certain Hodge theoretic condition, the $0$-th cohomological support loci of $K_X$ generates $\Pic^0(X)$. We then  show that the condition that the $0$-th cohomological support loci of $K_X$  generates $\Pic^0(X)$ can often be applied to prove the birationality of certain pluricanonical maps of $X$.
\end{abstract} 
 \section{Introduction}

In this article, we study the pluricanonical systems of irregular varieties.
 This topic was initiated in a series of article of Jungkai Chen and Hacon \cite{CH, CH2, CH3}.

Let $X$ be a smooth projective variety of general type with $q(X)=h^1(X, \mathcal O_X)>0$.
Fix the Albanese morphism $a_X: X\rightarrow A_X$ from $X$ to its Albanese variety $A_X$. We denote by $F$ a connected component of a general fiber of $a_X$. We say that $X$ is of Albanese fiber dimension $m$ if $\dim F=m$. The properties of the pluricanonical maps of $X$ are closely related to the properties of the pluricanonical systems of $F$. When $ F$ or $X$ is of low dimensions, the general picture is  well-understood.

\begin{theo}\label{A0}[J.A.Chen-Hacon \cite{CH}; Jiang-Lahoz-Tirabassi \cite{JLT}]
Let $X$ be a smooth projective variety of general type, of maximal Albanese dimension. Then the linear system $|mK_X+P|$ induces a birational map of $X$ for any integer $m\geq 3$ and $P\in\Pic^0(X)$.
\end{theo}

Note that this is an optimal result, which can be easily deduced from the properties of pluricanonical systems on curves of genus $2$. A similar result was proved by the author and Hao Sun when $F$ is a curve.

\begin{theo}\label{A1}[Jiang-Sun \cite{JS}]
Let $X$ be a smooth projective variety of general type, of Albanese fiber dimension $1$. Then the linear system $|mK_X+P|$ induces a birational map of $X$ for any integer $m\geq 4$ and $P\in\Pic^0(X)$.
\end{theo}
 It is worth noting that the structure of $$V^0(K_X):=\{P\in \Pic^0(X)\mid H^0(X, K_X\otimes P)\neq 0\}$$ plays an important role in the proof of both above theorems. We know from generic vanishing theory (see for instance \cite{GL} and \cite{S}) that $$V^j(K_X):=\{P\in \Pic^0(X)\mid H^j(X, K_X\otimes P)\neq 0\},$$for $j\geq 0$, is a union of torsion translates of abelian subvarieties of $\Pic^0(X)$. We say that \textit{ $V^j(K_X)$ generates $\Pic^0(X)$} if the irreducible components of $V^j(K_X)$ generate  $\Pic^0(X)$ as a group. One of the main points in the proof of Theorem \ref{A0} and \ref{A1} is to show that  $V^0(K_X)$ always  generates $\Pic^0(X)$ in both cases.
 
 The first main result of this paper is a general criteria for the property that $V^0(K_X)$ generates $\Pic^0(X)$.

 Recall that we say a smooth projective variety $V$ satisfies the infinitesimal Torelli condition if the natural cup product
 \begin{eqnarray*}
 H^1(V, T_V)\rightarrow \mathrm{Hom}(H^0(V, K_V), H^1(V, \Omega_V^{\dim V-1}))
 \end{eqnarray*}
 is injective. The infinitesimal Torelli condition is of course closely related with the variation of Hodge structures of $V$.

 Given a smooth projective variety of general type $V$. We know that the birational transformation group $\mathrm{Bir}(V)$ is a finite group and acts naturally on $H^0(V, K_V)$.
 \begin{theo}\label{main1}
 Assume that $X$ is of general type and $F$ satisfies the following conditions:
 \begin{itemize}
 \item[(C1)] the canonical model of $F$ is a smooth projective variety $V$ which satisfies the infinitesimal Torelli condition;
 \item[(C2)] $\mathrm{Bir}(F)$ acts faithfully on $H^0(F, K_F)$,
 \end{itemize}
 then $V^0(K_X)$ generates $\Pic^0(X)$.
 \end{theo}

This is a further extension of the corresponding results in \cite{CH3} and \cite{JS}. Indeed, it is well known that a non-hyperelliptic curve of genus $\geq 3$ satisfies both conditions in the assumption of Theorem \ref{main1}. One of the main ingredient of the proof of Theorem \ref{main1} is a decomposition theorem of  Hodge modules on abelian varieties proved in \cite{PPS}. Actually our proof works also for higher cohomology support loci $V^0(R^jf_*K_X)$ (see Theorem \ref{hypersurface}).

We also explore the relation between the structure of $V^0(K_X)$ and the properties of  the pluricanonical systems of $X$.

\begin{theo}\label{main2}Assume that $V^0(K_X)$ generates $\Pic^0(X)$ and $|K_F|$ induces a birational map of $X$. Then the linear system $|3K_X+P|$ induces a birational map of $X$ for $P\in\Pic^0(X)$ general.
\end{theo}

The conditions in Theorem \ref{main1} and Theorem \ref{main2} apply to varieties whose Albanese fibers $F$ has sufficiently positive canonical bundles.
It is well known that surfaces of general type does not satisfy the infinitesimal Torelli condition in general. Besides surfaces with $p_g=0$, there are surfaces with $p_g>0$, which does not satisfy the infinitesimal Torelli condition (see for instance \cite{To}). Hence Theorem \ref{main1} and Theorem \ref{main2} do not apply to varieties of Albanese fiber dimension $2$.  It is then  surprising to the author that by taking hyperplane sections, the structure of cohomological support loci of varieties of Albanese fiber dimension $1$ has implications on the pluricanonical systems of varieties of Albanese fiber dimension $2$.

\begin{theo}\label{fiber=2}Assume that $X$ is a smooth projective variety of general type and a connected component of a general fiber of the Albanese morphism of $X$ is a surface $F$ with $p_g(F)>0$ and $q(F)=0$. Then $|5K_X+P|$ induces a birational map of $X$ for all $P\in\Pic^0(X)$.
 \end{theo}
\begin{rema}
In \cite{CCCJ}, the authors shows that $|5K_X+P|$ always induces a birational map for any irregular $3$-fold $X$, which is an optimal result. The essential difficulty in that result is to deal with irregular $3$fold which is of Albanese fiber dimension $2$. Hence Theorem \ref{fiber=2} can be regarded as a partial higher dimensional generalization of the main result in \cite{CCCJ}. The proofs of the two results are quite different. The authors applied repeatedly Reider type arguments in \cite{CCCJ} and this kind of result is out of reach in higher dimensions.
\end{rema}

\subsection*{Acknowledgements}
Parts of the work was finished during the author's visit to  National Center for Theoretical Sciences in Taipei in January 2019. The author thanks the warm hospitality of Jungkai and Jhengjie. The author is supported by the Natural Science Foundation of China  (No.~11871155 and  No.~11731004), by the National Key Research and Development Program of China (No.~2020YFA0713200), and by the Natural Science Foundation of Shanghai (No.~21ZR1404500).

\section{The structure of $V^0(K_X)$ }

\subsection{Hodge Modules}\label{HodgeModule}
The proof of the main result relies heavily on the language of Hodge modules. We recall some results about Hodge Modules on abelian varieties. All results can be found in \cite{PPS}, where Pareschi, Popa, and Schnell applied the mechinary of Hodge modules to prove the general Chen-Jiang decomposition for higher direct image of canonical bundles on abelian varieties generalizing the main result in \cite{CJ}. Indeed, Pareschi, Popa, and Schnell proved the decomposition of certain Hodge modules which is crucial for our purpose.

Let $X$ be a complex manifold. We denote $\mathbf{HM}_{\mathbb R}(X, w)$ be the category of real Hodge module on $X$ of weight $w$, which is a semi-simple $\mathbb R$-linear abelian category,endowed with a faithful functor to the category of real perverse sheaves.
An object $M$ of $\mathbf{HM}_{\mathbb R}(X, w)$ consists of a regular holonomic left $\mathcal D_X$-module $\mathcal M$, a good filtration $F_{\cdot}M$ by coherent $\mathcal O_X$-modules, a perverse sheaf $M_{\mathbb R}$ with coefficients in $\mathbb R$ and an isomorphism $M_{\mathbb R}\otimes_{\mathbb R}\mathbb C\simeq \mathrm{DR}(\mathcal M)$. The support of $M$ is defined to be the support of the perverse sheaf $ M_{\mathbb R}$. One can also define polarizable Hodge modules. One can check \cite{PPS} or \cite{Sai} for more details about this category.

\begin{itemize}
\item[(A)]  Every object $M\in \mathbf{HM}_{\mathbb R}(X,w)$ admits a locally finite decomposition by strict support:
$$M\simeq \bigoplus_{i=1}^nM_i,$$
where each $M_i\in \mathbf{HM}_{\mathbb R}(X,w)$ has strict support equal to an irreducible analytic subvariety $Z_i\subset X$.
\item[(B)]  The category of polarizable real Hodge modules of weight $w$ with strict support $Z\subset X$ is equivalent to the category of generically defined polarizable real variations of Hodge structure of weight $w-\dim Z$ on $Z$.
\item[(C)]  Denote by $\mathbb R_X[\dim X]\in \mathbf{HM}_{\mathbb R}(X,\dim X)$ the polarizable real Hodge module corresponding to the constant real variation of Hodge structure on $X$. For any morphism $f:X\rightarrow A$ from $X$ to an abelian variety and an integer $j\geq 0$, $\mathcal H^ f_*\mathbb R_X[\dim X]$ is a polarizable real Hodge module of weight $j+\dim X$ on $A$. Let $M=(\mathcal M, F_{\cdot}\mathcal M, M_{\mathbb R})$ be the direct summand of $\mathcal H^j f_*\mathbb R_X[\dim X]$ with strict support $f(X)$. Then the first non-trivial grading piece of $\mathcal M$ is $R^jf_*\omega_X$.
\item[(D)] Let $M$ be as in $(C)$ above. Then we associate a complex polarizable Hodge module $(M\oplus M, J_M)$ as in \cite{PPS}, whose underlying perverse sheaf is simply $M\otimes_{\mathbb R}\mathbb C$. Then by \cite[Theorem 7.1 and Corollary 7.3]{PPS}, we know that $$(M\oplus M, J_M)\simeq \bigoplus_{j=1}^nq_j^{-1}(N_j, J_j)\otimes_{\mathbb C}\mathbb C_{\rho_j},$$ where $q_j: A\rightarrow A_j$ is a surjective morphism with connected fibers between abelian varieties,  $\rho_j\in \mathrm{Char}(A)$ is a unitary character, and $(N_j,J_j)\in \mathbf{HM}_{\mathbb C}(T_j, \dim X-\dim q_j)$ is a simple polarizable complex Hodge module with $\chi(T_j,N_j,J_j)>0$.
\item[(E)] Under the assumption of $(D)$, then by \cite[Theorem D]{PPS}, for any $k\in \mathbb Z$, $$\mathrm{gr}_k^F\mathcal M\simeq \oplus_{j=1}^nq_j^*\mathcal F_j\otimes L_j,$$ where $\cF_j$ is an M-regular sheaf on $T_j$ and $L_j$ is a torsion line bundle on $T$.

\end{itemize}
 
 \subsection{The proof of Theorem \ref{main1}}
 We work in a  more general setting.
 
\begin{theo}\label{hypersurface} Assume that $X$ is a  smooth projective variety of general type. Let $f: X\rightarrow A$ be a morphism from $X$ to an abelian variety and let $F$ be a connected component of a general fiber of $f$. Fix a positive integer $j\geq 0$. Assume that $F$ satisfies the following conditions:
 \begin{itemize}
 \item[(C1')] the canonical model of $F$ is a smooth projective variety $V$ which satisfies the following infinitesimal Torelli condition: the map induced by cup-product $$H^1(V, T_V)\rightarrow \mathrm{Hom}(H^j(V, K_V), H^{j+1}(V, \Omega_V^{\dim V-1}))
 $$ is injective;
 \item[(C2')] $\mathrm{Bir}(F)$ acts faithfully on $H^j(F, K_F)$,
 \end{itemize}
 then $V^0(R^jf_*\omega_X)$ generates $\Pic^0(X)$.
\end{theo}
 \begin{proof}
We argue by contradiction. Assume that the translates through the origin of the irreducible components of $V^0(R^jf_*\omega_X)$ generate a proper abelian subvariety $\PB$ of $\Pic^0(A)$. Consider the dual $B$ of $\PB$, we then have
\begin{eqnarray*}
\xymatrix{X\ar[d]_{f}\ar[dr]^g\\
A\ar[r]_q & B.}
\end{eqnarray*}

 Let $K$ be the fiber of of $q: A\rightarrow B$ over a general point $b\in B$ and $Y$ the corresponding fiber of $X\rightarrow B$. We then have a morphism $f_b: Y\rightarrow K$. Note that $R^jf_{*}\omega_X|_K\simeq R^jf_{b*}\omega_Y$.
We then consider the natural restriction morphism $\pi: \Pic^0(A)\rightarrow \Pic^0(K)$.
It is clear that $$\pi(V^0(R^jf_*\omega_X))\subset V^0(R^jf_{b*}\omega_Y),$$ since $H^0(A, R^jf_*\omega_X\otimes P)\neq 0$ implies that $H^0(K, R^jf_{b*}\otimes P|_K)\neq 0$ for   $\in\Pic^0(A)$. We claim that
$\pi(V^0(f_*\omega_X))=V^0(f_{b*}\omega_Y)$. Indeed, for  $Q\in V^0(R^jf_*\omega_Y)$. We choose $P\in\Pic^0(A)$ such that $P|_K=Q$. Then $$q_*(R^jf_{*}\omega_X\otimes P)$$ is non-trivial, since its fiber at $b$ is exactly $H^0(K, R^jf_{b*}\omega_Y\otimes Q)\neq 0$. On the other hand, it is known that $q_*(R^jf_{*}\omega_X\otimes P)$ is a GV sheaf by Hacon's criterion \cite{H} and Koll\'ar's vanishing theorem \cite{K}. Thus there exists $P'\in \Pic^0(B)$ such that $H^0(B, q_*(f_{*}\omega_X\otimes P)\otimes P')\neq 0$. Hence $$P\otimes P'\in V^0(R^jf_*\omega_X)$$ and $\pi(P\otimes P')=Q$.

By the construction of $\PB$, we then know that $\dim V^0(R^jf_{b*}\omega_Y)=0$ and hence $V^0(R^jf_{b*}\omega_Y)$ consists of finite many torsion points.  
Repeating the above argument, we can assume furthermore that $K$ is a simple abelian variety.
We now work with $f_b: Y\rightarrow K$. 

 Let $m=\dim Y$.  Since $\dim V^0(R^jf_{b*}\omega_Y)=0$, we conclude from the decomposition theorem \cite{PPS} that $R^jf_{b*}\omega_Y$ is a direct sum of torsion line bundles on $K$. In particular, $f_b$ is surjective.
 
 Let $M$ be the direct summand of $\mathcal H^jf_{b*}\mathbb R[m]$ with strict support $f_b(Y)$. Note that the first non-trivial grading piece of the D-module $\cM$ of $M$ is $R^jf_{b*}\omega_Y$.
 
 By Subsection \ref{HodgeModule} (D), we know that
 \begin{eqnarray*}(M\oplus M, J_M)&=& \mathcal H'\bigoplus \mathcal H''
 \\&=&\big(\bigoplus_jq^{-1}V_j\otimes \mathbb C_{\rho_j} \big)\bigoplus \mathcal H'',
  \end{eqnarray*}where $q: A\rightarrow \mathrm{Spec}\; \mathbb C$ is the constant morphism, $V_j$ are Hodge structures of weights $\dim F+j$, $\rho_j\in\mathrm{Char}(K)$  are unitary characters, and $\mathcal H''$ is a direct sum of simple complex polarizable Hodge modules with positive holomorphic Euler characteristic.
 The fact that $\dim V^0(R^jf_{b*}\omega_Y)=0$ implies that $R^jf_{b*}\omega_Y$ is indeed the first non-trivia grading piece of the underlying D-module of $\mathcal H'$. Indeed, since the holomorphic Euler characteristic of $\cH''$ is positive, any Hodge grading  piece of $\mathcal H''$ has positive holomorphic Euler characteristic (see \cite[Lemma 15.1]{PPS}).

  By Subsection \ref{HodgeModule} (B), $\mathcal H^jf_{b*}\mathbb R[m]$ corresponds to the variation of Hodge structure of $H^{j+\dim F}(Y_t)$, where $t\in K$ belongs to the smooth locus of $f_b$ and $Y_t$ is the fiber of $f_b$ over $t$. Then $\cH'$  is a polarizable variation of complex Hodge structure, whose fiber over a general point $t\in K$ is a sub-Hodge structure of $H^{j+\dim F}(Y_t, \mathbb C)$ containing $H^j(Y_t, K_{Y_t})$.  Note that $\cH'$ is trivial up to an \'etale cover of $K$.
  
Since $F$ is a connected component of $Y_t$, the composition of the Kodaira-Spencer map with the cup product of variation of Hodge structures:
  $$T_{K,t}\rightarrow H^1(F, T_F)\rightarrow \mathrm{Hom}(H^j(F, K_F), H^{j+1}(F, \Omega_F^{\dim F-1}))$$
  is zero.

  We now consider the relative canonical model of $f_b$: $$\wf_b: \wY\rightarrow K.$$ By the assumption, we know that a connected component of the fiber $\wf_b$ over $t$ is the the canonical  model $V$ of $F$. By assumption, $V$ is a smooth projective variety  which satisfies the infinitesimal Torelli condition (C1'). Note that $H^j(F, K_{F})=H^j(V, K_{V})$.
 and $H^{j+\dim V}(V, \mathbb C)\subset H^{j+\dim F}(F, \mathbb C)$ is a direct summand, we have the commutative diagram:
 \begin{eqnarray*}
 \xymatrix{T_{K, t}\ar@{=}[d]\ar[r]& H^1(F, T_F)\ar[d]\ar[r]& \mathrm{Hom}(H^j(F, K_{F}), H^{j+1}(F, \Omega_{F}^{\dim F-1}))\ar[d]\\
 T_{K, t}\ar[r]& H^1(V, T_V)\ar[r]& \mathrm{Hom}(H^j(V, K_{V}), H^{j+1}(V, \Omega_{V}^{\dim V-1})).}
 \end{eqnarray*}
 Thus the composition of maps $$T_{K, t}\rightarrow H^1(V, T_{V}) \rightarrow \mathrm{Hom}(H^j(V, K_{V}), H^{j+1}(V, \Omega_{V}^{\dim V-1}))$$ is also zero.

By assumption that $V$ satisfies the infinitesimal Torelli condition,  the Kodaira-Spencer map $T_{K, t}\rightarrow H^1(V, T_{V})$ induced by the family $\wf_b$ is thus zero. Hence $\wf_b$ is locally isotrivial around $t\in K$ and thus
the morphism $f_b: Y\rightarrow K$ is  birationally isotrivial.

In the following, we shall apply the condition (C2') to show that for a  birational isotrivial morphism $f_b: Y\rightarrow K$, $\dim V^0(R^jf_{b*}\omega_Y)=0$ is absurd.

We  assume without loss of generalities that $f_b: Y\rightarrow K$ is primitive, i.e. $f_b$ does not factor through \'etale covers of $K$. We consider the Stein factorization $$f_b: Y\xrightarrow{h_1}N\xrightarrow{h_2} K$$ of $f_b$, where after birational modifications we may assume that $N$ is smooth, $h_1$ is a fibration and $h_2$ is generically finite and surjective. 
Since $K$ is a simple abelian variety and $f_b$ is primitive, either $h_2$ is birational or $\deg h_2\geq 2$ and $N$ is of general type.

Note that $h_1$ is a birationally isotrivial fibration.  Since $F$ is of general type, its birational automorphism group is finite, we then take a Galois cover $\rho: M\rightarrow N$ such that the flat base change $M\times_N Y$ is birational to $M\times F$.  If $\rho$ factors through an abelian \'etale cover of $N$ induced by an isogeny $\tilde{K}\rightarrow K$, we replace $f_b: Y\rightarrow K$ by $\tilde{f_b}: \tilde{Y}:=Y\times_K\tilde{K}\rightarrow \tilde{K}$. Note that $V^0(\tilde{f_b}_*\omega_{\tilde{Y}})$ still consists of finitely many points.  Therefore, without loss of generalities, we assume that  $\rho: M\rightarrow N$ does not factor through  abelian \'etale covers of $N$ induced by  isogenies of $K$. 

 Let $G$ be the Galois group of the cover $\rho$. Taking an $G$-equivariant resolution of singularities, we may assume that $M$ is smooth.  We then consider  the canonical decomposition $$\rho_*\omega_M=\omega_N\bigoplus_{1_G\neq \chi\in \mathrm{Irr}(G)}\mathcal V_{\chi}$$ with respect to the $G$-action, where $\mathrm{Irr}(G)$ denotes the set of irreducible representations of $G$. 
 
 We claim that $h_{2*}\mathcal V_{\chi}$ is M-regular on $K$ for any $1_G\neq \chi\in \mathrm{Irr}(G)$. If not, by the Chen-Jiang decomposition theorem (see \cite{PPS}), $h_{2*}\mathcal V_{\chi}$ has a numerically trivial line bundle $P$ as a direct summand. Since $h^{\dim M}(M, h_{2*}\rho_*\omega_M)=h^{\dim M}(K, h_{2*}\omega_N)=1$, $P$ is a non-trivial torsion line bundle on $K$. Since $f_b$ is primitive, $h_2^*P$ is also a non-trivial torsion line bundle on $N$. But $h^{\dim M}(M, \omega_M\otimes \rho^*h_2^*P^{-1})\neq 0$, thus $\rho^*h_2^*P=\mathcal O_M$, which implies that $\rho$ factors through an abelian \'etale cover  of $N$ induced by an isogeny of $K$, which is a contradiction.

 We are now ready to deduce a contradiction. We observe that $Y$ is birational to the diagonal quotient $(M\times F)/G$, where $G$ acts on $F$ via an injective homomorphism $G\rightarrow \mathrm{Bir}(F)$.  

Since $G$ acts naturally on $H^j(F, K_F)$, we then consider the representation decomposition $$H^j(F, K_F)=\bigoplus_{\chi\in\mathrm{Irr}(G)}V_{\chi}.$$ By the condition (C2), $G\subset\mathrm{Bir}(F)$ acts faithfully on $H^j(F, K_F)$.
 Hence for some $\chi_0\neq 1_G$, $V_{\chi_0}$ is non-zero.

  From  the commutative diagram:
  \begin{eqnarray*}
  \xymatrix{
  M\times F\ar[d]\ar[r]& (M\times F)/G\sim Y\ar[d]^{h_1}\ar[dr]^{f_b}\\
  M\ar[r]^{\rho}& M/G=N\ar[r]^{h_2} & K,}
  \end{eqnarray*}
  we see that \begin{eqnarray*}&&R^jf_{b*}\omega_Y=h_{2*}R^jh_{1*}\omega_Y=h_{2*}R^jh_{1*}\omega_{(M\times F)/G}\\&&=h_{2*}(\rho_*\omega_M\otimes H^j(F, K_F))^{G}=\bigoplus_{\chi, \chi'\in \mathrm{Irr}(G)}(\mathcal V_{\chi}\otimes_{\mathbb C} V_{\chi'
 })^G.
\end{eqnarray*}

 It is known from the character theory of representation of finite groups that for $\chi_0\neq 1_G$, there exists $\chi_1\neq 1_G\neq \in\mathrm{Irr}(G)$ such that $(\mathcal V_{\chi_1}\otimes_{\mathbb C} V_{\chi_0 })^G$ is a non-trivial direct summand of $\mathcal V_{\chi_1}\otimes_{\mathbb C} V_{\chi_0
 }$. Since $\cV_{\chi_1}$ is M-regular, $(\mathcal V_{\chi_1}\otimes_{\mathbb C} V_{\chi_0 })^G$ is also M-regular on $K$.
  Thus $R^jf_{b*}\omega_Y$ has an M-regular sheaf as a direct summand. This contradicts the assumption that $\dim V^0(R^jf_{b*}\omega_Y)=0$.
 \end{proof}

 \begin{rema}
 Perhaps the most useful case of Theorem \ref{hypersurface} is when $j=0$, because the  conditions (C1) and (C2) have been extensively studied.

 All smooth complete intersections of general type in  projective spaces satisfy $\mathrm{(C1)}$ (see \cite{Pe} or \cite{U}). Moreover, there is a fairly general criterion for varieties of general type satisfying $\mathrm{(C1)}$ proved in \cite{Ki}: assume that $K_F=L^{\otimes m}$ for some $m\geq 1$, the base locus of $|L|$ is of codimension $\geq 2$, and $h^0(F, \Omega_F^{\dim F-1}\otimes L)\leq h^0(L)-2$, then $$H^1(F, T_F)\rightarrow \mathrm{Hom}(H^0(F, K_F), H^1(F, \Omega_F^{\dim F-1}))$$ is injective.

 For $\mathrm{(C2)}$, assuming that a finite subgroup $G\subset \mathrm{Bir}(F)$ acting trivially on $H^0(F, K_F)$, then the canonical map of $F$ factors through the quotient map by $G$. This is of course impossible if the canonical system $|K_F|$ defines a birational map of $F$. Moreover, if $|K_F|$ defines a generically finite map of $F$, then the condition that the canonical map of $F$ factors through the quotient map by $G$ implies that $p_g(F)=p_g(F/G)$ (see \cite[Th\'eor\`eme 3.1]{B}), which is very restrictive.

 In conclusion, we conclude that the assumptions (C1) and (C2) are satisfied by a large number of varieties of general type, including all smooth complete intersections of general type in projective spaces. 
 \end{rema}

The proof of Theorem \ref{hypersurface} provides more information.

\begin{coro}\label{simpleHodge}  Under the assumption of Theorem \ref{hypersurface}. Assume  the sub-Hodge structure of $H^{j+\dim F}(F, \mathbb C)$ containing $H^j(F, K_F)$ is simple and nonzero and $f$ is a fibration onto its image,   $R^jf_*\omega_X$ is M-regular.
\end{coro}
\begin{proof}We have seen by Theorem \ref{hypersurface} that $V^0(R^jf_*\omega_X)$ generates $\Pic^0(A)$. If $R^jf_*\omega_X$ is not M-regular, the Chen-Jiang decomposition for  $R^jf_*\omega_X$ has at least $2$ direct summands.
By Subsection \ref{HodgeModule}, we know that the direct summands of $R^jf_*\omega_X$ comes from the first grading pieces of the decompositions of the Hodge module $M$.  Restricting to a general fiber of $f$, this implies that the sub-Hodge structure of $H^{j+\dim F}(F, \mathbb C)$ containing $H^j(F, K_F)$ is not simple, which is a contradiction.
\end{proof}

 \section{Tricanonical maps of $X$ with $K_F$ sufficiently positive}
 \begin{theo}\label{3K}Let $f:X\rightarrow A$ be a morphism from a projective variety to an abelian variety. Let $F$ be an irreducible component of a general fiber of $f$. Assume that $V^0(f_*\omega_X)$ generates $\Pic^0(A)$ and $|K_F|$ induces a birational map of $F$, then
 \begin{itemize}

 \item[(1)] the linear system $|3K_X+f^*P|$ induces a birational map of $X$ for $P\in\Pic^0(A)$ general;
 \item[(2)] the linear system $|mK_X+f^*P|$ induces a birational map for $m\geq 4$ and all $P\in\Pic^0(A)$.
 \end{itemize}
 \end{theo}
 \begin{proof}

 We first note that $(2)$ is easy. By the assumption that $V^0(f_*\omega_X)$ generates $\Pic^0(A)$ and Tirabassi's argument in \cite{T}, we know that $f_{*}(\omega_X^2\otimes \cI_x)$ is M-regular for a general point $x\in X$. It follows immediately from standard arguments that $|mK_X+f^*P|$ induces a birational map of $X$ for all $m\geq 4$ and $P\in \Pic^0(X)$.

 The proof of $(1)$ consists of several steps. Without loss of generalities, we will assume that $f$ is primitive and $f(X)$ generates $A$.

 First of all, by the decomposition theorem \cite{PPS}, we write
 \begin{eqnarray}\label{dec}f_*\omega_X=\bigoplus_{p_{B_i}: A\rightarrow B_i}\bigoplus_jp_{B_i}^*\cF_{B_i, j}\otimes P_{B_i, j}^{-1},\end{eqnarray}
where $p_{B_i}$ are surjective morphisms with connected fibers between abelian varieties, $\cF_{B_i, j}$ are M-regular sheaves on $B_i$, $P_{B_i, j}$ are torsion line bundles.  Then $$V^0(\omega_X)=\bigcup_{i, j}\big(P_{B_i, j}+\Pic^0(B_i)\big),$$ where $p_{B_i}$ and $j$ runs through the direct summands in the decomposition formula (\ref{dec}).

Let $p:=\prod_{B_i} p_{B_i}: A\rightarrow \prod_iB_i$. By the assumption that $V^0(f_*\omega_X)$ generates $\Pic^0(A)$, we know that $p$ is finite onto its image. Let $K$ be the image of $p$. Then $p: A\rightarrow K\hookrightarrow \prod_iB_i$ is an isogeny between $A$ and $K$.

Since $M$-regular sheaves are continuously globally generated, we conclude that the twisted evaluation map
\begin{eqnarray}\label{surjective}\bigoplus_{U_{ij}}\bigoplus_{Q\in  U_{ij}}H^0(A, f_*\omega_X\otimes Q)\otimes Q^{-1}\rightarrow f_*\omega_X
 \end{eqnarray}is surjective, where $U_{ij}\subset P_{B_i, j}+\Pic^0(B_i)$ are open subsets. \\

\noindent {\bf Step 1. $|3K_X+f^*P|$ for any  $P\in\Pic^0(A)$, separates two general points on the same general fiber of $f$.}\\

Let $x, y\in X$ be two general points in the same general fiber $X_a$ of $f$ for $a\in A$. Since $V^0(f_*\omega_X)$ contains positive dimensional components, $h^0(X, 2K_X)=h^0(X, 2K_X+f^*Q)\geq 2$ for any $Q\in \Pic^0(A)$.

Fix $P\in \Pic^0(X)$.

If for some $B_i$ and $P_{B_i, j}$ as in (\ref{dec}) and $Q\in P_{B_i, j}+\PB_i$ general, the linear system $$|2K_{X}+f^*P-f^*Q|$$ separates $x$ and $y$, then for $Q\in P_{B_i, j}+\PB_i$ general, there exists $D_1\in |2K_X+f^*P-f^*Q|$ such that $x\in D$ and $y\notin D$. Since $Q\in P_{B_i, j}+\PB_i$ general,  there exists $D_2\in |K_X+f^*Q|$ containing neither $x$ nor $y$, then $D_1+D_2\in |3K_X+f^*P|$ separating  $x$ and $y$.

 We then assume that for all $B_i$ and $P_{B_i, j}$ and $Q\in P_{B_i, j}+\Pic^0(B_i)$, the corresponding linear system $$|2K_{X}+f^*P-f^*Q|$$ cannot separate $x$ and $y$. Since $|K_F|$ induces a birational map of $F$, we conclude from (\ref{surjective}) that there exist positive integers $N_{ij}$, $Q_{ijk}\in P_{B_i, j}+\Pic^0(B_i)$ general for $1\leq k\leq N_{ij}$, sections $s_{i j k}\in H^0(X, K_X+ f^*Q_{ijk})$  such that the section $$s:=\sum_{ij}\sum_{1\leq k\leq N_{ij}}s_{ijk}\mid_{X_a}\in H^0(X_a, K_{X_a})$$ separating $x$ and $y$, i.e. $x\in D_s$ and $y\in D_s$, where $D_s$ is the corresponding divisor of $s$ on $X_a$. We  choose $s'_{ijk}\in H^0(X, 2K_X+ f^*P-f^*Q_{ijk})$ which  take the same nonzero value at both $x$ and $y$ in an affine charts of $X_a$ containing both $x$ and $y$. Note that $s_{ilk}'$ exists since $\varphi_{ijk}(x)=\varphi_{ijk}(y)$, where $\varphi_{ijk}$ is the rational map induced by $|2K_X+f^*P-f^*Q_{ijk}|$.

 Then $$\sum_{ij}\sum_{1\leq k\leq N_{ij}}s_{ijk}\cdot s'_{ijk}\in H^0(X, 3K_X+ f^*P)$$ separating  $x$ and $y$. Hence the twisted tricanonical map induced by $|3K_X+f^*P|$ separates two general points on the same  general fiber of $f$. \\

\noindent {\bf Step 2. $|3K_X+f^*P|$ induces a generically finite map from $X$ onto its image  for all $P\in \Pic^0(A)$.}\\

  Assume the contrary, there exists a curve $C$ through a general point of $X$ contracting by the map $|3K_X+f^*P|$. Then the rank of the restriction map $$H^0(X, 3K_X+f^*P)\rightarrow H^0(C, (3K_X+f^*P)|_C)$$ is $1$. By Step 1.,  $ f|_C: C\rightarrow A$ is generically finite onto its image. By the assumption, there exists $p_B$ appearing in the decomposition formula (\ref{dec}) such that $p_B\circ f|_C: C\rightarrow B$ is generically finite onto its image. Let $T=P+\Pic^0(B)$ be a subset of $V^0(f_*\omega_X)$. Then for $Q\in T$ general, the image $V_Q$ of the restriction map $$H^0(X, K_X+f^*Q)\rightarrow H^0(C, (K_X+f^*Q)|_C)$$ is not zero. We also denote by $V_Q'$ the image of the restriction map
 $$H^0(X, 2K_X+f^*P-f^*Q)\rightarrow H^0(C, (2K_X+f^*P-f^*Q)|_C).$$ We observe that the image of $\Pic^0(B)\rightarrow \Pic^0(C)$ is positive dimensional.

 The image of the restriction map $$H^0(X, 3K_X+f^*P)\rightarrow H^0(C, (3K_X+f^*P)|_C)$$ contains the map of natural multiplication map $$\bigcup_{Q\in T}V_Q\otimes V_{-Q}'\rightarrow H^0(C, (3K_X+f^*P)|_C). $$ Thus the rank of the restriction map $$H^0(X, 3K_X+f^*P)\rightarrow H^0(C, (3K_X+f^*P)|_C)$$ is bigger than $1$, which is a contradiction. \\

 \noindent {\bf Step 3. For $P\in \Pic^0(A)$ general, let $x, y\in X$ be two general points such that $p\circ f(x)\neq p\circ f(y)\in K$. Then the map $|3K_X+f^*P|$ separates $x$ and $y$.}\\

 For a general point $x\in X$,  for each component $P_{B_i, j}+\PB_i$ appearing in (\ref{dec}), $x\notin \mathrm{Bs}(|K_X+f^*Q|)$ for $Q\in P_{B_i, j}+\PB_i$ general. Then $x\notin \mathrm{Bs}(|2K_X+f^*Q|)$ for all $Q\in \PT_{i,j}:=2P_{B_i, j}+\PB_i$. Considering $g_i:=p_{B_i}\circ f: X\rightarrow B_i,$ we know that $g_{i*}(\omega_X^2\otimes f^*Q)$ is IT$^0$ for all $Q\in \Pic^0(A)$ (see \cite{LPS}). Hence from the short exact sequence $$0\rightarrow g_{i*}(\omega_X^2\otimes f^*Q\otimes\cI_x)\rightarrow g_{i*}(\omega_X^2\otimes f^*Q)\rightarrow \mathbb C_x\rightarrow 0,$$ we conclude that $g_{i*}(\omega_X^2\otimes f^*Q\otimes \cI_x)$ is again IT$^0$ on $B_i$, for all $Q\in\PT_{ij}$.
Note that there exists an open subset $(x, Q)\in\mathcal U_i\subset X\times\Pic^0(X)$ parametrizing the flat family of sheaves $g_{i*}(\omega_X^2\otimes f^*Q\otimes\cI_x)$ on $B_i$. Being IT$^0$ is an open condition, hence there exists a non empty open subset $\mathcal U_i'$ of $\mathcal U_i$ parametrizing IT$^0$ sheaves. Let $\mathcal V$ be the intersection of all $\mathcal U_i'$ for all $p_{B_i}$ appearing in the decomposition formula. Then $\mathcal V$ is again a non-empty open subset of $X\times \Pic^0(A)$. Then for a general $P\in \Pic^0(A)$, there exists an open subset $W$ of $X$ such that $g_{i*}(\omega_X^2\otimes f^*P\otimes\cI_x)$ is IT$^0$ for all $g_{i}$ and all $x\in W$.
Hence for $P\in\Pic^0(A)$ general, there exists an open subset $W'$ of $X$ such that $$g_{i*}(\omega_X^2\otimes f^*P\otimes f^*P_{B_i, j}^{-1}\otimes\cI_x)$$ is IT$^0$ for all $i, j$ and all $x\in W'$.

We now take $W''$ to be the open subset of $W'$ such that for any $x\in W''$, $x\notin \mathrm{Bs}(|K_X+f^*P|)$ for any $i, j$ and $P\in P_{B_i, j}+\PB_i$ general.
Then for $x, y\in W''$ and $p(x)\neq p(y)$, there exists $p_{B_i}: A\rightarrow B_i$ such that $g_i(x)\neq g_{i}(y)$. Since $g_{i*}(\omega_X^2\otimes f^* P\otimes f^*P_{B_i, j}^{-1}\otimes\cI_x)$ is IT$^0$, for $Q\in \Pic^0(B_i)$ general, there exists a divisor $D\in |2K_X+f^*P-f^*P_{B_i,j}+f^*Q|$ separating $x$ and $y$ and hence the linear system $|3K_X+f^*P|$  also separates $x$ and $y$.\\

\noindent {\bf Step 4. Conclusion.}\\

We now argue that for $P\in \Pic^0(A)$ general, $|3K_X+f^*P|$ induces a birational map of $X$. Assume the contrary, let $\varphi: X\dashrightarrow Y\subset \mathbb P^N$ be the map induced by $|3K_X+f^*P|$, where $Y$ is the image of $\varphi$. Note that there exists a map $f_Y: Y\dashrightarrow K$ such that we have the commutative diagram:
\begin{eqnarray*}
\xymatrix{
X\ar@{.>}[r]^{\varphi}\ar[d]^f & Y\ar@{.>}[d]^{f_Y}\\
A\ar[r]^{p} & K. }
\end{eqnarray*}
Indeed, for $x\in Y$ general, $\varphi^{-1}(y)$ consists of finitely many points and by Step 1 and Step 2, $p\circ f(\varphi^{-1}(y))$ is a single point of $K$. Hence we can define the above map $Y\dashrightarrow K$.

After birational modifications of $X$ and $Y$, we may assume that $Y$ is smooth, $\varphi$ is an morphism and $Y\rightarrow K$ is also a morphism. After replacing $K$ by a suitable \'etale cover of it, we may alao assume that $Y\rightarrow K$ is primitive. Since $\varphi$ maps a general fiber of $f$ birationally onto its image, we conclude that $\varphi: X\rightarrow Y$ is birationally equivalent to the \'etale morphism $Y\times_KA\rightarrow Y$ induced by base change of $p: A\rightarrow K$. But then $$\varphi_*\omega_X^3=\bigoplus_{Q\in \ker p^*}\omega_Y^3\otimes Q.$$ We then take $P'\in \Pic^0(K)$ such that $p^*P'=P$. Then $$H^0(X, 3K_X+f^*P)=\bigoplus_{Q\in \ker p^*}H^0(Y, 3K_Y+f_Y^*P'+f_Y^*Q).$$ Note that $H^0(Y, 3K_Y+f_Y^*P'+f_Y^*Q)\neq 0$ for all $Q$. Hence it is absurd that $\varphi$ is birationally equivalent to the base change of $p$.

 \end{proof}

 Combining Theorem \ref{main1} and Theorem \ref{3K}, we have
 \begin{coro} Let $f: X\rightarrow A$ be a morphism from a smooth projective variety to an abelian variety. Let $F$ be a connected component of a general fiber of $f$. If $F$ satisfies the birationally infinitesimal Torelli condition $\mathrm{(C1)}$ and $|K_F|$ induces a birational map of $F$, then
 \begin{itemize}

 \item[(1)] the linear system $|3K_X+f^*P|$ induces a birational map of $X$ for $P\in\Pic^0(A)$ general;
 \item[(2)] the linear system $|mK_X+f^*P|$ induces a birational map for $m\geq 4$ and all $P\in\Pic^0(A)$.
 \end{itemize}
 \end{coro}

\section{$K_F$ is not so positive}
The goal of this section is to prove Theorem \ref{main2}, which partially generalize the main result of \cite{CCCJ}.

\begin{theo}\label{regular-surface}
Let $f: X\rightarrow A$ be the Albanese morphism of a smooth projective variety of general type $X$ to its Albanese variety $A$. Let $F$ be a connected component of a general fiber of $f$.
Assume that $F$ is a surface with $p_g(F)>0$ and $q(F)=0$, $|5K_X+P|$ induces a birational map of $X$ for all $P\in\Pic^0(A)$.
\end{theo}

 \subsection{Set-up}

In this subsection we denote by $f: X\rightarrow A$ the Albanese morphism of $X$   and let $F$ be a connected component of a general fiber of $f$ over its image.  For a line bundle $L$ on $X$, we define $V^0(L):=\{Q\in\Pic^0(A)\mid h^0(X, L\otimes Q)\neq 0\}$. For $Q\in\Pic^0(A)$, we can  regard it as a line bundle on $X$ or on $A$, or just a point of $\Pic^0(A)$ when there is no confusion.

 Let $m_2\geq m_1\geq 2$ be integers such that $$ |m_1K_F|\neq \emptyset$$ and $|m_2K_F|$ induces a birational map of $F$. We explore some general criteria of the birationality of $|(m_1+m_2)K_X|$.
 \begin{rema}
 If $F$ is a  surface of general type, then we can take $m_1=2$ and $m_2=3$ (see \cite{BPV}), unless its minimal model $F_0$ satisfies either  $(K_{F_0}^2, p_g(F_0))=(1, 2)$ or $(2, 3)$. In the latter cases, the tricanonical map of $F$ is of degree $2$.
 
\end{rema}

We recall that for $m\geq 2$ and $h^0(F, mK_F)\neq 0$, by the Theroem of Lombardi-Popa-Schnell \cite{LPS}, $f_*\omega_X^m$ is IT$^0$.
 
 The starting point is the following lemma, which is essentially  \cite[Propositiom 3.2 and Proposition 3.3]{CCCJ}.

\begin{prop}\label{criteria1}
Let $x\in X$ be a general point and $\cI_x\subset \mathcal O_X$ the ideal sheaf of this point.
Assume that either $f_*(\omega_X^{m_1}\otimes \cI_x)$ or  $f_*(\omega_X^{m_2}\otimes \cI_x)$ is M-regular, then $|(m_1+m_2)K_X+P|$ induces a birational map of $X$ for all $P\in\Pic^0(A)$.
\end{prop}
\begin{proof}
The argument is the same as in \cite[Propositiom 3.2 and Proposition 3.3]{CCCJ}. For reader's convenience we repeat the argument here.
We need to show that $|(m_1+m_2)K_X+P|$ separates two general points $x$ and $y$.

We first assume that
$f_{*}(\omega_X^{m_2}\otimes \cI_x)$ is M-regular and let $F$ be the component of $a_X^{-1}a_X(x)$ containing $x$. Then by \cite{PP}, M-regular sheaves are continuously globally generated, which implies
that for any open subset $U$ of $\Pic^0(A)$, the evaluation map $$\bigoplus_{Q\in U}H^0(X, \cI_x(m_2K_X)\otimes Q)\rightarrow H^0(F, \cI_x(m_2K_F))$$
is surjective. Since $f_{*}\omega_X^{m_1}$ is IT$^0$, $|m_1K_X+P-Q|$ is nonzero for any $Q\in \Pic^0(A_X)$. Hence $|(m_1+m_2)K_X+P|$ separates $x$ and $y$ if $y\notin F$. By the assumption that $|m_2K_F|$ induces a birational map of $F$, we conclude as well that $|(m_1+m_2)K_X+P|$ separates $x$ and $y$ if $y\in F$.

We then assume that $f_{*}(\omega_X^{m_1}\otimes \cI_x)$ is M-regular. The same argument shows that $|(m_1+m_2)K_X+P|$ separates $x$ and $y$ either when $y\notin F$ or when $y\in F$ and $x, y\in F$ can be separated by $|m_1K_F|$. If $x, y\in F$ cannot be separated by $|m_1K_F|$, then take a section $f\in H^0(F, m_2K_F)$ such that its corresponding divisor contains $x$ and does not contain $y$. Since $f_{*}\omega_X^{m_2}$ is IT$^0$, for $U\subset \Pic^0(A_X)$ open, there exists $P_i\in U$, $1\leq i\leq N$ and $f_i\in H^0(X, m_2K_X+P+P_i)$ such that $$\sum_i f_i\mid_F=f.$$ We can pick $g_i\in H^0(X, m_1K_X-P_i)$ which    take the same nonzero value at both $x$ and $y$ in a suitable affine open neighborhood of $x $ and $y$. Then the divisor of  $F:=\sum_if_ig_i\in H^0(X,(m_1+m_2)K_X+P)$  separates $x$ and $y$.
\end{proof}

\begin{coro}
Assume that $V^0(K_X)$ generates $\Pic^0(X)$. Then $|(m_2+2)K_X+P|$ induces a birational map of $X$ for all $P\in\Pic^0(A)$.
\end{coro}
\begin{proof}
Since $V^0(K_X)$ generates $\Pic^0(X)$, we know that $f_{*}(\omega_X^2\otimes\cI_x)$ is M-regular by \cite{T}. We then conclude by Proposition \ref{criteria1}.
\end{proof}

\subsection{The Abel-Jacobi maps induced by fixed components of pluricanonical systems.}\label{AJ}

We assume in this subsection that for some $m\geq 2$ and $h^0(F, mK_F)>0$, $f_*(\omega_X^m\otimes\cI_x)$ is not M-regular for $x\in X$ general. Since \cite{BLNP}, we understand that the fixed component of $|mK_X+P|$ moves along with $P$ and the induced Abel-Jacobi map is quite interesting. We extend some related results of \cite{BLNP} and \cite{CCCJ} in this setting.
 
  Following the same argument in \cite[Section 4]{CCCJ}, we consider the closed subscheme of $X\times \Pic^0(A)$:
$$B_m:=\{(x, P)\mid x\in \mathrm{Bs} |mK_X+P|\}\subset X\times \Pic^0(A).$$ 
Note that the defining ideal of $B_m$ is given by the relative evaluation map of $\omega_X^m\otimes (f\times Id_A)^*\cP_A$, where $\cP_A$ is the Poincar\'e bundle on $A\times \Pic^0(A)$.

By the assumption that  $f_*(\omega_X^{m}\otimes \cI_x)$ is not M-regular,   $\dim B_{m}=\dim X+\dim \Pic^0(A)-1$ and  $B_{m}$ dominates $X$ via the first projection (see for instance \cite[Lemma 4.2]{CCCJ}). Let $\cD_m$  be the union of the components of $B_{m}$  which are divisors of $X\times \Pic^0(A)$ and dominate $X$ via the first projection.

For $P\in \Pic^0(A)$ general, we have
\begin{eqnarray}\label{linearsystem}
|mK_X+P|=|M_P|+D_P+E_m,
 \end{eqnarray}
where $D_P$  is the fiber of $\cD_m\rightarrow \Pic^0(A)$ over the point $P$, $E$ is the common fixed divisor, and $|M_P|$ has no fixed divisor.

We note that $\dim |M_{P}|=p_{m}(X)-1\geq 1$ if $p_m(X)\geq 2$ or $M_P=0$ when $p_m(X)=1$.

Since $f: X\rightarrow A$ is the Albanese morphism of $X$, we denote by $\Phi_m: \Pic^0(A)\rightarrow \Pic^0(A)$ the  Abel-Jacobi map induced by $\cD$, namely we have $$\mathcal O_X(D_{P_1}-D_{P_2})=\Phi_m(P_1)-\Phi_m(P_2)$$ for $P_1$, $P_2\in \Pic^0(A)$ general and $\Phi_m(0)=0$. We denote by $\PB_m\subset \Pic^0(A)$ the image of $\Phi$ and $\PC_m\subset \Pic^0(A)$ the kernel of $\Phi_m$.

\begin{lemm}\label{easy}\cite[Lemma 5.1]{BLNP}
We have $\Phi_m\circ \Phi_m=\Phi_m$ and hence $\Pic^0(A)=\PB_m+\PC_m\simeq \PB_m\times \PC_m$.
\end{lemm}
\begin{proof}
 Let $P_1, P_2\in\Pic^0(A)$ be general, then
$$M_{P_1}+D_{P_2}+E_m\sim mK_X+P_1+\Phi_{m}(P_2-P_1).$$
Since $f_{*}\omega_X^{m}$ is IT$^0$, $h^0(X, mK_X+Q)=p_{m}(X)$ for any $Q\in \Pic^0(A)$.
Thus $D_{P_2}+E_m$ is the fixed part of $|mK_X+P_1+\Phi_{m}(P_2-P_1)|$. Thus $\Phi_{m}\circ \Phi_{m}(P_2-P_1)=\Phi_{m}(P_2-P_1)$ and we have $\Phi_{m}\circ \Phi_{m}=\Phi_{m}$. Hence $\Pic^0(A)= \mathrm{Im}(\Phi_{m})+\mathrm{Im}(\mathrm{Id}_{A}-\Phi_{m})= \mathrm{Im}(\Phi_{m})+\mathrm{ker}(\Phi_{m})=\PB_{m}+\PC_{m}\simeq \PB_{m}\times \PC_{m}$.
\end{proof}
Since $\cD_m$ dominates $\Pic^0(A)$, $\PB_m$ is a positive dimensional abelian subvariety of $\Pic^0(A)$. Given $P\in\Pic^0(A)$, we write $P_B=\Phi_m(P) \in \PB_m$ and $P_C=P-P_B\in \PC_m$.

Let $D=D_P-\Phi_m(P)$ for $P\in\Pic^0(A)$ general. Note that $D$ is  independent of the choice of general $P$ by the definition of $\Phi_m$.  Consider the composition of morphisms $$X\times \Pic^0(A)\xrightarrow{f\times \mathrm{Id}_{\Pic^0(A)}}A\times\Pic^0(A)\xrightarrow{\mathrm{Id}_A\times \Phi_m} A\times\Pic^0(A),$$ we let $\cP_B$ on $X\times \Pic^0(A)$ be the pull-back of the Poincar\'e bundle $ A\times\Pic^0(A)$ via this composition of morphisms. Then by the see-saw principle, there exists a divisor $\Theta$ on $\Pic^0(A) $ such that we have an isomorphism of line bundles $$\cO_{X\times\Pic^0(A)}(\cD_m)\simeq p_1^*\cO_X(D)\otimes p_2^*\cO_{\Pic^0(A)}(\Theta)\otimes\cP_B,$$ where $p_1$ and $p_2$ are respectively the projections from $X\times\Pic^0(A)$ to $X$ and $\Pic^0(A)$.

For $P\in\Pic^0(A)$ general, $D_P$ is linear equivalent to $D+P_B$. Let $M=mK_X-D-E_m$. Then $M_P$ is linear equivalent to $M+P_C$ for $P\in\Pic^0(A)$ general by (\ref{linearsystem}).

\begin{lemm}\label{L-property}
We have  $V^0(M+E_m)=V^0(M)=\PC_m$ and $h^0(X, M+P_C+E_m)=h^0(X, M+P_C)=p_{m}(X)$ for any $P_C\in\PC_m$. Similarly, $V^0(D+E_m)=V^0(D)=\PB_m$ and $h^0(X, D+P_B+E_m)=h^0(X, D+P_B)=1$ for any $P_B\in \PB_m$.
  \end{lemm}
\begin{proof}
We know that for $P\in \Pic^0(A)$, $h^0(X, \cO_X(mK_X+P))=p_{m}(X)$ is constant. For $P$ general, we have $$|mK_X+P|=|M_P|+D_{P}+E_m,$$ and $M_P\sim M+P_C$ and $D_P\sim D+P_B$. Thus $V^0(M+E_m)\supset V^0(M)\supset \PC_m$ and $V^0(D+E_m)\supset V^0(D)\supset \PB_m$. By semicontinuity and the fact that $h^0(X, \cO_X(mK_X+P))$ is constant, 
$h^0(X, M+P_C+E_m)=h^0(X, M+P_C)=p_{m}(X)$ for any $P_C\in\PC_m$  and $h^0(X, D+P_B+E_m)=h^0(X, D+P_B)=1$ for any $P_B\in \PB_m$. 

In particular, $\cD$ is flat over $\Pic^0(A)$ and $D_P$ is the unique effective divisor linearly equivalent to $D+P_B$ for any $P\in\Pic^0(A)$.
Then, for any $P\in\Pic^0(A)$, $$|mK_X+P|=|M+P_C|+D_{P_B}+E_m.$$

It suffices to show that $V^0(M+E_m)=\PC_m$ and $V^0(D+E_m)=\PB_m$.
We argue by contradiction.

 Assume that there exists $Q=Q_B+Q_C\in\Pic^0(A)$ with $Q_B$ non-trivial such that $h^0(X, M+Q+E_m)>0$. Then for $P_B\in \PB_m$, we have
$$|M+Q+E_m|+D_{P_B}\subset |mK_X+Q+P_B|=|M+Q_C|+D_{Q_B+P_B}+E_m.$$ Thus
$D_{Q_B+P_B}\preceq \mathrm{Fix}|M+Q+E_m|+D_{P_B}$. Let $P_B\in \PB_m$ be general, we then have $D_{Q_B+P_B}\preceq D_{P_B}$, which is absurd.

Assume that there exists $Q=Q_B+Q_C\in\Pic^0(A)$ with $Q_C$ non-trivial such that $h^0(X, D+Q+E_m)>0$. Let $\tilde{D}_Q\in |D+Q+E_m|$. Then for $P_C\in\PC_m$, we have 
$$\tilde{D}_Q+|M+P_C|\subset |mK_X+Q+P_C|.$$ By dimension reason, the two linear systems are equal. Thus $$\tilde{D}_Q+|M+P_C|= |mK_X+Q+P_C|=|M+P_C+Q_C|+D_{Q_B}+E_m.$$
For $P_C\in \PC_m$ general,  both $|M+P_C|$ and $|M+P_C+Q_C|$ have no fixed part. Thus $\tilde{D}_Q=D_{Q_B}+E_m$ is the fixed part of $ |mK_X+Q+P_C|$, which is absurd.
\end{proof}

\begin{prop}\label{surj1}
The natural morphism
\begin{eqnarray*}&&f_{*}\mathcal O_X(M)\otimes f_{*}\mathcal O_X(D)\xrightarrow{\cdot E_{m}} f_{*}\mathcal O_X(mK_X) 
\end{eqnarray*}
is surjective.
Hence restricting to $F$, we have the surjective map
\begin{eqnarray*}&& H^0(F, M|_F)\otimes H^0(F, D|_F)\xrightarrow{\cdot E_m|_F} H^0(F, mK_F) 
\end{eqnarray*}
\end{prop}
\begin{proof}
We know that $f_{*}\omega_X^{m}$ is IT$^0$ and hence is M-regular. Thus by \cite{PP}, the evaluation map
$$\bigoplus_{Q\in \Pic^0(A)} H^0(A, f_{*}\mathcal O_X(mK_X)\otimes Q)\otimes Q^{-1}\rightarrow f_{*}\mathcal O_X(mK_X)$$ is surjective.

For $Q\in \Pic^0(X)$, we write  $Q=Q_B+ Q_C$ with $Q_B\in \PB_{m}$ and $Q_C\in \PC_{m}$. By Lemma \ref{L-property}, we have $$|mK_X+Q|=| M+Q_C|+ D_{Q_B}+E_m.$$
 Hence $$H^0(f_{*}\mathcal O_X(M)\otimes Q_C)\otimes H^0(f_{*}\mathcal O_X(D)\otimes Q_B)\xrightarrow{\cdot E_m} H^0(f_{*}\mathcal O_X(mK_X)\otimes Q) $$ is an isomorphism.

The composed surjective map
\begin{eqnarray*}
\bigoplus_{Q\in \Pic^0(A)}\big(H^0(f_{*}\mathcal O_X(M)\otimes Q_C)\otimes H^0(f_{*}\mathcal O_X(D)\otimes Q_B)\big)\otimes Q^{-1}\rightarrow \\ f_{*}\mathcal O_X(mK_X)
\end{eqnarray*}
factors through the map \begin{eqnarray}\label{surj}f_{*}\mathcal O_X(M)\otimes f_{*}\mathcal O_X(D)\rightarrow f_{*}\mathcal O_X(mK_X).\end{eqnarray} Hence (\ref{surj}) is also  surjective. Thus considering the above map restricted on $F$, we have the surjective map $$H^0(F, M|_F)\otimes H^0(F, D|_F)\rightarrow H^0(F, mK_F).$$ 
\end{proof}

We also need to consider the irreducible components of $\cD_m$.
Write $\cD_m=\sum_{j=1}^s a_j\tilde{\cD}_{j}$ as the sum of irreducible components and let $$\Phi_{mj}: \Pic^0(A)\rightarrow \Pic^0(A)$$ be the Abel-Jacobi map induced by $\tilde{\cD}_{j}$, namely $\Phi_{mj}(P)=\cO_X(\tilde{D}_{jP}-\tilde{D}_{j})$, where $\tilde{D}_{jP}$ is the fiber of $\tilde{\cD}_j$ over $P\in\Pic^0(A)$ and $\tilde{D}_{j}$ is the fiber over the neutral point of $\Pic^0(A)$. Let $\PB_{mj}$ be the image of $\Phi_{mj}$.  Since $\tilde{\cD}_{j}$ dominates $\Pic^0(A)$, $\dim\PB_{mj}>0$ for each $1\leq j\leq s$.
\begin{prop}\label{D-property}
We have
\begin{itemize}
\item[(1)] $\cD$ is reduced, i.e. $a_j=1$ for $1\leq j\leq s$ and $\cO_X(\tilde{D}_{j_1}-\tilde{D}_{j_2})\notin \Pic^0(A)=\Pic^0(X)$ for $1\leq j_1\neq j_2\leq s$;
\item[(2)]$\Phi_{mj}\circ \Phi_{mj}=\Phi_{mj}$ and $\Phi_{mj_1}\circ\Phi_{mj_2}=0$ for $j_1\neq j_2$ and $\PB_i\simeq \PB_{m1}\times\cdots\times \PB_{ms}$;
 
\end{itemize}
\end{prop}
\begin{proof}
Note that for any $P\in\PB_{mj}$, $|\tilde{D}_j+P|$ is non-empty by the definition of $\Phi_{mj}$. Thus if $a_j>1$, we have $\dim |a_j\tilde{D}_j+P|>1$, which contradicts the fact that $h^0(X, D+P)=1$ for any $P\in\Pic^0(A)$. Thus $\cD$ is reduced.

Similarly, if $\tilde{D}_{j_1}\sim \tilde{D}_{j_2}+Q$ for some $Q\in\Pic^0(A)$, then $h^0(X, \tilde{D}_{j_1}+\tilde{D}_{j_2}+Q)>1$, which is absurd.

We now consider $(2)$. For $P_1, P_2\in\Pic^0(A)$ general, $$|mK_X+P_1+\Phi_{mj}(P_2-P_1)|\supset |M_{P_1}|+\sum_{k\neq j}\tilde{D}_{kP_1}+\tilde{D}_{jP_2}+E_m.$$ By the dimension reason, the two linear systems are the same. Since $\cO_X(\tilde{D}_{j_1}-\tilde{D}_{j_2})\notin \Pic^0(A)$ for $j_1\neq j_2$, we conclude that $$\tilde{D}_{k (P_1+\Phi_{mj}(P_2-P_1))}=\tilde{D}_{kP_1}$$ for $k\neq j$ and $$\tilde{D}_{j (P_1+\Phi_{mj}(P_2-P_1))}=\tilde{D}_{jP_2}.$$ Therefore 
$\Phi_{mj}\circ \Phi_{mj}=\Phi_{mj}$ and $\Phi_{mj_1}\circ\Phi_{mj_2}=0$ for $j_1\neq j_2$. Hence $\PB_i\simeq \PB_{m1}\times\cdots\times \PB_{ms}$.

\end{proof}
 
\subsection{$F$ is a surface with $p_g(F)>0$ and $q(F)=0$}

We now assume that $F$ is a surface of general type and this subsection is devoted to the proof of Theorem \ref{fiber=2}.

 By assumption $p_g(F)>0$, hence $f_{*}\omega_X$ is a nonzero sheaf on $A$. By \cite{PPS}, we know that $f_{*}\omega_X$ admits the Chen-Jiang  decomposition. We  denote by $\pi: X\rightarrow X_0$ the contraction from $X$ to its canonical model $X_0$. Since the fibers of $\pi$ are rationally chain connected (see \cite[Corollary 1.5]{HM}), we have the commutative diagram:
\begin{eqnarray*}
\xymatrix{
X\ar[r]^{\pi}\ar[dr]_{f}& X_0\ar[d]^{f_0}\\
& A.}
\end{eqnarray*}
Note that a connected component of a general fiber of $f_{0}$ is then the canonical model of the corresponding component of the fiber of $f$. Moreover, for any divisor $E$ which is $\pi$ exceptional, we have
\begin{eqnarray}\label{=}H^0(X, mK_X+P)=H^0(X, mK_X+E+P),
\end{eqnarray}
for any $m\geq 0$ and $P\in \Pic^0(A)$.

We have a small variant of Proposition \ref{criteria1}.
\begin{lemm} Assume that $f_{*}(\omega_X^2\otimes \cI_x)$ or $f_{*}(\omega_X^3\otimes \cI_x)$ is M-regular for $x\in X$ general, then $|5K_X+P|$ induces a birational map of $X$ for any $P\in\Pic^0(A)$.
\end{lemm}
\begin{proof}
We know that $p_2(F)>1$ for any surfaces $F$ of general type. Moreover, if $F$ is not birational to a $(1, 2)$ surface or $(2, 3)$ surface, then  $|3K_F|$ induces a birational map of $F$ (see \cite[Section 7]{BPV}). We then conclude by Proposition \ref{criteria1}.

If $F$ is birational to a $(1, 2)$ surface or $(2, 3) $ surface, then $|3K_F|$ induces a degree $2$ map. We can conclude by Proposition \ref{criteria1} and the same argument in \cite[Proposition 3.1 and Proposition 3.2]{CCCJ}.
\end{proof}

From now on, we will assume that both $$f_{*}(\omega_X^2\otimes \cI_x)$$ and $$f_{*}(\omega_X^3\otimes \cI_x)$$ are not M-regular. We now apply the results in   Subsection \ref{AJ}. 

For $m=2$ and $3$, we consider (\ref{linearsystem}) and write respectively for $P\in\Pic^0(A)$
\begin{eqnarray*} |2K_X+P|=|M_P|+D_P+E_2;\\
|3K_X+P|=|L_P|+V_P+E_3,
\end{eqnarray*}
where $|M_P|$ and $|L_P|$ have no fixed components when $P$ is general and $D_P$ and $V_P$ are respectively the fibers of $\cD_2$ and $\cD_3$ over $P$.

There exist divisors $M$ and $D$ on $X$ such that $D_P\sim D+P_B$ and $|M_P|=|M+P_C|$, where $P=P_B+P_C$ and $P_B=\Phi_2(P)\in\PB_2$ and $P_C\in\PC_2$. 

Similarly, there exist divisors $L$ and $V$ on $X$ such that $V_Q\sim V+Q_B$ and $|L_Q|=|L+Q_C|$, where $Q=Q_B+Q_C$ with $Q_B=\Phi_3(Q)\in\PB_3$ and $Q_C\in\PC_3$. 

We can study the relation between $\cD_2$ and $\cD_3$ with the help of $|K_X+Q_0|$ for 
$Q_0\in V^0(K_X)$.
Note that  $$|K_X+Q_0|+|2K_X+P|\subset |3K_X+P+Q_0|.$$ Hence $$|K_X+Q_0|+|M_P|+D_P+E_2\subset |L_{P+Q_0}|+V_{P+Q_0}+E_3.$$ For $P\in\Pic^0(A)$ general, comparing the fixed parts of the above linear systems, we have
\begin{eqnarray}\label{relation2}V_{P+Q_0}+E_3\preceq D_P+\mathrm{Fix}|K_X+Q_0|+E_2.\end{eqnarray} Since each component of $\cD_2$ and $\cD_3$ dominates $\Pic^0(A)$, we hav3
\begin{eqnarray}\label{relation1}V_{P+Q_0}\preceq D_P
\end{eqnarray} for $P\in \Pic^0(A)$ general and hence for all $P\in\Pic^0(A)$. 
Hence $(\mathrm{Id}_X\times t_{Q_0})_*\cD_3$ is  a sub-divisor of $\cD_2$, where $t_{Q_0}$ is the translation by $Q_0$ on $\Pic^0(A)$. Let $\cD_2=\cS+(\mathrm{Id}_X\times t_{Q_0})_*\cD_3$ and let $S_P$ be the fiber of $\cS$ over $P\in \Pic^0(A)$.

We then have $D_P=V_{P+Q_0}+S_P$.
 
 By Lemma \ref{D-property}, we know that the image $\PB_3$ of $\Phi_3$ is an abelian subvariety of the image $\PB_2$ of $\Phi_2$. Moreover, there exists a abelian subvariety of $\PK$, which is the image of the map $\Phi_2-\Phi_3$, such that $\PB_2=\PB_3+\PK\simeq \PB_3\times \PK$ and $\PC_3=\PC_2+\PK\simeq \PC_2\times \PK$.
\begin{lemm}\label{V0}  Any translate through the origin of an irreducible component of $V^0(K_X)$ is contained in $ \PC_3$.
\end{lemm}
\begin{proof} For $Q_0\in V^0(K_X)$ and $P\in\Pic^0(A)$, we have $V_{P+Q_0}\preceq D_P$. Fix $P$, we see that any translate through the origin of an irreducible component of $V^0(K_X)$ is contained in  $\Phi_3=\PC_3$.
 
\end{proof}

 \begin{lemm}\label{control-V0}
 We then have
\begin{itemize}
\item[(1)] Any translate through the origin of an irreducible component of $V^0(K_X+M+S)$ is contained in $\PC_3$.
\item[(2)] $V^0(K_X+V)\subset V^0(K_X)+\PB_3$.
\end{itemize}
\end{lemm}
\begin{proof}
By (\ref{relation2}) and (\ref{relation1}), $|K_X+Q_0|+|M_P|+S_P+E_2\subset |L_{P+Q_0}|+E_3$  for $Q_0\in V^0(K_X)$ and $P\in\Pic^0(A)$. Because $V^0(L+E_3)=\PC_3$ by Lemma \ref{L-property}, $V^0(K_X+M +S)$ is contained in  $Q_0+\PC_3$. We conclude the proof of (1) by Lemma \ref{V0}.

Fix $P\in V^0(K_X+V)$ and let $Q\in\Pic^0(A)$ be general. Let $T$ be a general member of $|M_Q|+S_Q$. We see that $(K_X+V+P)+T+E_2$ is linearly equivalent to $3K_X+P+Q$. Hence $$|K_X+V+P|+T+E_2\subset |3K_X+P+Q|=|L_{P+Q}|+V_{P+Q}+E_3.$$
Since $Q$ general, no irreducible components of $T$ is a sub-divisor of $V_{P+Q}+E_3$, thus $T\preceq |L_{P+Q}|$ and  $$V_{P+Q}\preceq |K_X+V+P|+E_2.$$ Thus $P+V-V_{P+Q}=P-\Phi_3(P+Q)\in V^0(K_X+E_2)$. By the following Lemma \ref{ex} and (\ref{=}), $V^0(K_X+E_2)=V^0(K_X)$ and hence $$V^0(K_X+V) \subset V^0(K_X)+\PB_3.$$
 \end{proof}
 \begin{lemm}\label{ex}
 $E_2$ is $\pi$-exceptional and $E_2|_F$ is $\pi_F$-exceptional, where $\pi_F: F\rightarrow F_0$ is the contraction from $F$ to its canonical model $F_0$.
 \end{lemm}
 \begin{proof}
Since $f_{*}\omega_X^2$ is M-regular, thus $$\bigoplus_{P\in\Pic^0(A)}H^0(f_{*}\omega_X^2\otimes P)\otimes P^{-1}\rightarrow f_{*}\omega_X^2$$ is surjective. Let $F$ be a connected component of a general fiber of $f$ and $F_0$ its canonical model. Hence $\bigoplus_{P\in\Pic^0(A)}H^0(X, 2K_X+P)\rightarrow H^0(F, 2K_F)$ is surjective. Since $|2K_X+P|=|M_P|+D_P+E_2$, we conclude that $E_2\mid_F$ is contained in the fixed component of $|2K_F|$. On the other hand, since $p_g(F)>0$, the bicanonical system of the minimal model of $F$ is base point free and hence $E_2|_F$ is an exceptional divisor for the contraction $F\rightarrow F_0$. Hence $E$ is $\pi$-exceptional.
 \end{proof}

 \begin{lemm}\label{q=0} Assume moreover that $q(F)=0$, then for $P\in\Pic^0(A)$ general,
 \begin{itemize}
  \item[(1)] $M_P|_F$ is trivial and hence $(D_{P}+E_2)|_F$ is linearly equivalent to $2K_F$;
 \item[(2)]   We have $\PB_3=\PB_2$ and $\PC_3=\PC_2$ and hence $\cD_2=(\mathrm{Id}_X\times t_{Q_0})_*\cD_3$ for $Q_0\in V^0(K_X)$;
 \item[(3)] $V^0(K_X)\subset \PC_2$.
 
 \end{itemize}
 \end{lemm}

 \begin{proof}

Let $Y$ be an irreducible component of a general member of $|M_P|$ or an irreducible component of $S_P$ for $P$ general. Then considering the short exact sequence:
$$0\rightarrow K_X\rightarrow K_X+Y\rightarrow K_Y\rightarrow 0,$$
we see that $V^0(K_Y)\subset V^0(K_X+Y)\cup V^1(K_X)$. By the assumption that $q(F)=0$, we know $R^1f_{*}\omega_X=0$ and hence by generic Vanishing (see \cite{H}) $$V^1(K_X)=V^1(f_{*}\omega_X)\subset V^0(f_{*}\omega_X)=V^0(K_X).$$ Hence $V^0(K_Y)\subset V^0(K_X+Y)$. By Lemma \ref{control-V0}, we conclude that any translate through the origin of an irreducible component of $V^0(K_Y)$ is contained in $ \PC_3$. Then for any resolution of singularities $\rho: \tilde{Y}\rightarrow Y$, we have $V^0(\rho_*\omega_{\tilde{Y}})\subset V^0(K_Y)\subset \PC_3$ \footnote{Since $Y$ is Gorenstein, $\rho_*\omega_{\tilde{Y}}$ is sub-sheaf of $K_Y$.}.
Since a deformation of $\tilde{Y}$ covers $X$, $\tilde{Y}$ is of general type, by the main Theorem of \cite{JS}, a general fiber of $f|_Y\circ \rho: \tilde{Y}\rightarrow A$ is not of dimension $1$. We conclude that both $M_P|_F$ and $S_P|_F$ are trivial. Hence $(D_P+E_2)|_F$ is linearly equivalent to $2K_F$. Since $D_P=V_{P+Q_0}+S_P$ for $Q_0\in V^0(K_X)$, $E_2$ is $\pi$-exceptional by Lemma \ref{ex}, and $S_P|_F$ are trivial, there exists an irreducible component $N$ of $V_{P+Q_0}$  whose general fiber over $f(N)$ is a curve.
We then consider
$$0\rightarrow K_X\rightarrow K_X+N\rightarrow K_N\rightarrow 0$$ and apply the same argument as above.  By  Lemma \ref{control-V0},
 we know that $V^0(K_N)\subset V^0(K_X)+\PB_3$.
 Note that the family of $N$  covers $X$. Hence any desingularization $\tilde{N}$ of $N$ is of general type and a general fiber of the morphism $\tilde{N}\rightarrow f(N)\subset  A$ is a curve. Hence by the main result of \cite{JS}, $V^0(K_N)$ generates $\Pic^0(A)$.

 Because  each translate through the origin of an irreducible component of $V^0(K_X)$ is contained in $ \PC_3$ and $\Pic^0(A)=\PB_3+\PC_3$,  the translates through origin of irreducible components of $V^0(K_X)$ generates $\PC_3$.

 Assume that $V^0(K_X)=\bigcup_i(Q_i+\PT_i)$, where $Q_i$ is a torsion point and $\PT_i$ is an abelian subvariety of $\PC_3$. For any positive dimensional component $Q_i+\PT_i$ and $P\in \PT_i$ general, 
we write $$|K_X+Q_i+P|=|M_{iP}|+N_{iP}+E,$$ where $E$ is the common fixed divisor and $|M_{iP}|$ has no fixed part.
 Hence for $P\in \PT_i$ general, 
  $$|M_{iP}|+|M_{i (-P)}|+N_{iP}+N_{i(-P)}+2E\subset |2K_X+2Q_i|=|M_{2Q_i}|+D_{2Q_i}+E_2,$$
 we see that $M_{iP}+M_{i (-P)}+N_{iP}+N_{i(-P)}\preceq M_{2Q_i}$. Thus $\PT_i\subset \PC_2$. Since these $\PT_i$ generates $\PC_3$, we have $\PC_2=\PC_3$ and $\PB_2=\PB_3$. Therefore  by Lemma \ref{D-property}, $\cS=0$ and $\cD_2=(\mathrm{Id}_X\times t_{Q_0})_*\cD_3$ for $Q_0\in V^0(K_X)$.
 
 Finally, recall that $V_{P+Q_0}\preceq D_P$ for any $Q_0\in V^0(K_X)$ and $P\in\Pic^0(A)$.
 We now know that $V$ and $D$ are algebraically equivalent. Thus fix $P$, $V_{P+Q_0}=D_P$ is the same divisor for any $Q\in V^0(K_X)$. Thus $V^0(K_X)\subset \PC_2$.

  \end{proof}
\begin{theo}
Assume that $p_g(F)>0$ and $q(F)=0$, $|5K_X|$ induces a birational map of $X$.
\end{theo}
\begin{proof}
Note that by the Chen-Jiang decomposition theorem,  the natural twisted evaluation map
$$\bigoplus_{Q\in V^0(K_X)}H^0(A, f_{*}\omega_X\otimes Q)\otimes Q^{-1}\rightarrow f_{*}\omega_X$$ is surjective, which implies the surjectivity of the restriction map
\begin{eqnarray}\label{image}\bigoplus_{Q\in V^0(K_X)}H^0(X, K_X+Q)\rightarrow H^0(F, K_F).
\end{eqnarray}

Note that for any $Q\in V^0(K_X)$, we have $$|K_X+Q|+|K_X+Q|\subset |2K_X+2Q|=|M_{2Q}|+D_{2Q}+E_2.$$ We know by Lemma \ref{V0} and Lemma \ref{q=0} that $V^0(K_X)\subset \PC_2$. Note that $\PC_2=\mathrm{ker}(\Phi_2)$. Hence $D_{2Q}=D_{2P_0}$ is fixed  for all $Q\in V^0(K_X)$. Moreover, $M_{2Q}|_F$ is trivial by Lemma \ref{q=0}. Hence the image of sum of evaluation map (\ref{image}) is $1$-dimensional. We have $h^0(F, K_F)=1$.

On the other hand, by Proposition \ref{surj1}, the map $$H^0(F, L|_F)\otimes H^0(F, V|_F)\xrightarrow{\cdot E_3|_F} H^0(F, 3K_F)$$ is surjective. Recall that $D=V_{Q_0}$ and $(D+E_2)|_F$ is linear equivalent to $2K_F$, where $E_2|_F$ is $\pi_F$-exceptional. 
Thus 
$L|_F\sim K_F+E_2|_F-E_3|_F$. We have $h^0(F, V|_F)=p_2(F)$ and $h^0(F, L|_F)\leq h^0(F, K_F)=1$.  This implies that $p_3(F)\leq p_2(F)$, which is absurd.
\end{proof}

 \end{document}